\documentclass[11pt]{article}
\usepackage{amssymb}
\usepackage{amsmath}
\usepackage{amsthm}
\usepackage{graphicx}
\usepackage[margin=0.5in]{geometry}
\usepackage{bm}
\usepackage{amsthm}
\usepackage{amscd}
\usepackage{amsbsy}
\usepackage{amsmath}
\usepackage{amsfonts}
\usepackage{amssymb}
\usepackage{calligra}
\usepackage{xcolor}
\usepackage{float}
\usepackage[T1]{fontenc}
\usepackage{multirow}
\usepackage{mathrsfs}
\usepackage{mathtools}
\usepackage{graphicx}
\usepackage{epstopdf}
\usepackage{subfigure}
\usepackage{booktabs}
\usepackage{listings}
\usepackage{longtable}
\usepackage{latexsym}
\usepackage{arydshln}
\usepackage{enumerate}
\usepackage{epsfig}
\usepackage{cite}
\usepackage{verbatim}
\usepackage{overpic}
\usepackage{abstract}
\allowdisplaybreaks[4]

\def\hs{\hspace{0.3cm}}

\newtheorem {theorem*}{Theorem}
\newtheorem {theorem} {Theorem}

\newtheorem{lemma}{Lemma}
\newtheorem{corollary}{Corollary}

\numberwithin{equation}{section}
\numberwithin{lemma}{section}
\numberwithin{theorem}{section}
\numberwithin{proposition}{section}
\numberwithin{corollary}{section}

\usepackage{diagbox}
\usepackage[colorlinks,
            CJKbookmarks=true,
            linkcolor=blue,
            anchorcolor=red,
             urlcolor=blue,
            citecolor=blue
            ]{hyperref}

\begin{document}
\arraycolsep=1pt

\title{\Large\bf  Products of Toeplitz and Hankel Operators
on Fock-Sobolev Spaces
\footnotetext{\hspace{-0.35cm}
\endgraf{\it E-mail: yiyuanzhang@e.gzhu.edu.cn (Yiyuan Zhang)
}
\endgraf \hspace{1.1cm} {\it guangfucao@163.com (Guangfu Cao)
}
\endgraf \hspace{1.1cm} {\it helichangsha1986@163.com (Li He)
}
\endgraf This work was partially supported by the National Natural Science Foundation of China (12071155, 11871170).
The first author was partially supported by the Innovation Research for the Postgraduates of Guangzhou University (2020GDJC-D08).
}
}
\author{Yiyuan Zhang, Guangfu Cao, Li He\thanks{Corresponding author}\\
\small\em School of Mathematics and Information Science,
Guangzhou University,\\ \small \em  Guangzhou,  Guangdong 510006, China
}
\date{ }
\maketitle

\vspace{-0.8cm}

\begin{center}
\begin{minipage}{16cm}\small
{\noindent{\bf Abstract} \quad In this paper, we investigate the boundedness of  Toeplitz product  $T_{f}T_{g}$ and Hankel product  $H_{f}^{*} H_{g}$ on Fock-Sobolev space for two polynomials $f$ and $g$ in $z,\overline{z}\in\mathbb{C}^{n}$. As a result, the boundedness of Toeplitz operator $T_{f}$ and Hankel operator $H_{f}$ with the polynomial symbol $f$ in $z,\overline{z}\in\mathbb{C}^{n}$ is characterized.

\endgraf{\bf Mathematics Subject Classification (2020).}\quad Primary 47B35; Secondary 30H20.
\endgraf{\bf Keywords.}\quad Toeplitz product, Hankel product, Fock-Sobolev space.
}
\end{minipage}
\end{center}

\section{Introduction}
\ \ \ \  Let $\mathbb{C}^{n}$ be the Euclidean space of complex dimension $n$ and $dv$ be the Lebesgue measure on $\mathbb{C}^{n}$. For $z=\left(z_{1}, \ldots, z_{n}\right)$ and  $w=\left(w_{1}, \ldots, w_{n}\right)$ in $\mathbb{C}^{n}$,
we denote
$$
z \cdot \overline{w}=\sum_{j=1}^{n} z_{j} \overline{w}_{j}, \quad|z|=(z \cdot \bar{z})^{1 / 2}.
$$

The Fock space $F^{2}$ consists of all entire functions $f$ on $\mathbb{C}^{n}$
such that
$$
\|f\|_{2}=\left(\frac{1}{\pi^{n}} \int_{\mathbb{C}^{n}}|f(z)|^{2} \mathrm{e}^{-|z|^{2}} d v(z)\right)^{\frac{1}{2}}<\infty.
$$

Let $\mathbb{N}$ be the set of nonnegative integers. For any multi-index $\alpha=(\alpha_{1},\ldots,\alpha_{n})\in \mathbb{N}^{n}$ and $z\in \mathbb{C}^{n}$, we write
$$|\alpha|=\alpha_{1}+\cdots+\alpha_{n}, \quad \alpha !=\alpha_{1} ! \cdots \alpha_{n} !, \quad \partial^{\alpha}=\partial_{1}^{\alpha_{1}} \cdots \partial_{n}^{\alpha_{n}},\quad z^{\alpha}=z_{1}^{\alpha_{1}}\cdots z_{n}^{\alpha_{n}},$$
where $\partial_{j}$ denotes the partial derivative with respect to the $z_{j}$.

For any $m\in\mathbb{N}$ , the  Fock-Sobolev space $F^{2,m}$ consists of
all entire  functions $f$ on $\mathbb{C}^{n}$ such that
$$\|f\|_{2,m}:=\sum_{|\alpha| \leq m}\|\partial^{\alpha}f\|_{2}<\infty.$$

The Fock-Sobolev space was introduced by Cho and Zhu in \cite{ChoZhu2012}, where they proved that $f\in F^{2,m}$ if and only if the function $z^{\alpha}f (z)$ is in $F^{2}$ for all multi-indexes $\alpha$ with $|\alpha| = m$, which allows us to introduce the equivalent norm on $F^{2,m}$:
$$
\|f\|_{2, m}=\left(\omega_{n, m} \int_{\mathbb{C}^{n}} \left| f(z) \right|^{2} \left|z\right|^{2m}e^{-|z|^{2}}d v(z)\right)^{\frac{1}{2}},
$$
where
$$
\omega_{n, m}=\frac{(n-1) !}{\pi^{n} \Gamma(m+n)}
$$
is a normalizing constant such that the constant function 1 has norm 1 in
$F^{2,m}$.


For any $z\in \mathbb{C}^{n}$,
Let
$$dV_{m}(z):=\omega_{n,  m}|z|^{2m} e^{-|z|^{2}}dv(z).$$

Denote $L^{2}_{m}$ by the space of Lebesgue measurable functions $f$ on $\mathbb{C}^{n}$ so that the function $ f(z)\in L^{2}(\mathbb{C}^{n}, dV_{m})$. It is well-known that the space $L^{2}_{m}$ is a Hilbert space with the inner product
$$
\langle f, g\rangle_{m}=\int_{\mathbb{C}^{n}} f(z) \overline{g(z)} d V_{m}(z).
$$

It is clear that the Fock-Sobolev space $F^{2,m}$ is a closed subspace of $L^{2}_{m}$. Let $P_{m}$ be the orthogonal projection from $L^{2}_{m}$ to $F^{2,m}$, that is
$$
P_{m} f(z)= \int_{\mathbb{C}^{n}} f(w) K_{m}(z, w)d V_{m}(z),
$$
where $K_{m}(z,w)$ is the reproducing kernel of  $F^{2,m}$.

For a Lebesgue measurable function $f$ on $\mathbb{C}^{n}$ such that $fK_{m}(z,\cdot)$ are in  $L^{2}(\mathbb{C}^{n},dV_{m})$ for all $z\in\mathbb{C}^{n}$,
the Toeplitz operator with symbol $f$ on $F^{2,m}$ is defined by $$T_{f} g=P_{m}(fg),$$
and the Hankel operator $H_{f}$ with symbol $f$ is given by
$$H_{f} g=\left(I-P_{m}\right)(fg),$$
where $I$ is the identity operator on $L^{2}_{m}$.



The original Toeplitz product problem was raised by Sarason in \cite{Sarason1994}, to ask whether one can  give a characterization for the pairs of outer functions $g,h$ in  the Hardy space $H^{2}$ such that the operator $T_{g}T_{\overline{h}}$ is bounded on $H^{2}$.
The famous Sarason's conjecture on this problem has attracted the attention of some mathematical researchers in operator theory. This problem was partially solved on the Hardy space of the unit circle in \cite{Zheng1996}, on the Bergman space of the unit disk in \cite{StroethoffZheng1999}, on the Bergman space of the polydisk in \cite{StroethoffZheng2003} and on the Bergman space of the unit ball in \cite{Park2006,StroethoffZheng2007}. Unfortunately, Sarason's conjecture was eventually proved to be false, both on the Hardy space and the Bergman space, see \cite{AlemanPottReguera2017,Nazarov1997} for counterexamples. However, in \cite{ChoParkZhu2014,BommierYoussfiZhu2017}, the Sarason's conjecture was  proved to be true on the Fock space, and in this setting, the explicit forms of the symbols $f$ and $g$ were given. Although the boundedness of a single Toeplitz operator on Fock space is still an open problem, some progress has been made in
Toeplitz products and Hankel products. Ma, Yan, Zheng and Zhu \cite{MaYanZhengZhu2019} gave a sufficient but not necessary condition on bounded Hankel product $H_{\overline{f}}^{*} H_{\overline{g}}$ for $f,g$ in the Fock space.
Yan and Zheng \cite{YanZheng2020} characterized bounded Toeplitz product $T_{f}T_{g}$ and Hankel product $H_{f}^{*} H_{g}$ on Fock space for two polynomials $f$ and $g$ in $z$, $\overline{z}\in\mathbb{C}$.
Inspired by these work, we study the boundedness of Toeplitz product $T_{f}T_{g}$ and Hankel product $H_{f}^{*} H_{g}$ on $F^{2,m}$ for two polynomials $f, g\in\mathcal{P}$, where
\begin{align*}
\mathcal{P}:=\left\{\prod_{s=1}^{n}\left(\sum_{\beta_{s}\leq k_{s}}\sum_{\gamma_{s}\leq l_{s}}a_{\beta_{s}\gamma_{s},s}z_{s}^{\beta_{s}}\overline{z}_{s}^{\gamma_{s}}\right):k_{s},l_{s}\in\mathbb{N},\ z_{s}\in\mathbb{C}\ \text{and}\ a_{\beta_{s}\gamma_{s},s}\ \text{are}\ \text{constants}\right\}.
\end{align*}

Our main results can be stated as follows.
\begin{theorem}\label{TP}
Let $f$ and $g$ be two polynomials in $z,\overline{z}\in\mathbb{C}^{n}$. Then the Toeplitz product $T_{f}T_{g}$ is bounded on $F^{2,m}$ if and only if both $f$ and $g$ are constants.
\end{theorem}
\begin{theorem}\label{HP}
Let $f$ and $g$ be two polynomials in $z,\overline{z}\in\mathbb{C}^{n}$. Then the Hankel product $H^{*}_{f}H_{g}$ is bounded on $F^{2,m}$ if and only if at least one of the following statements holds:

$(1)$ $f$ is holomorphic.

$(2)$ $g$ is holomorphic.

$(3)$ $n = 1$ and there exist two holomorphic polynomials $f_{1}$ and $g_{1}$ such that
\begin{align*}
f=f_{1}+a\overline{z},\quad g=g_{1}+b\overline{z},
\end{align*}
where $a,b$ are constants and $z,\overline{z}\in\mathbb{C}$.
\end{theorem}
We would like to mention that all the conclusions for the Fock-Sobolev space $F^{2,m}$ in this paper are consistent with the results in \cite{YanZheng2020} when $m = 0 $ and $n = 1$, but the boundedness characterization of Hankel product for $n\geq 2$ is essentially different from $n=1$ and all the results for $m\geq 1$ are new.

The layout of the paper is as follows. In Section \ref{section2} we give the proof of characterizations of bounded Toeplitz product $T_{f}T_{g}$  on $F^{2,m}$. In section \ref{section3} we give the proof of characterizations of bounded Hankel product $H^{*}_{f}H_{g}$.

In what follows, denote by $\chi_{E}$ the characteristic function of a measurable set $E$.  We say a multi-index $\alpha=(\alpha_{1},\ldots,\alpha_{n})\in \mathbb{N}^{n}$ tends to $\infty$ if each component $\alpha_{i}$ tends to $\infty$. For two arbitrary sequences $A_{\alpha}$ and $B_{\alpha}$ depending on multi-index $\alpha=(\alpha_{1},\ldots,\alpha_{n})$, we use the notation $A_{\alpha}\sim B_{\alpha}$ to denote the relationship:
\begin{align*}
\lim_{\alpha\rightarrow \infty}\frac{A_{\alpha}}{B_{\alpha}}=C,
\end{align*}
where $C$ is a positive constant independent of $\alpha$.

Recall the  Stirling's formula is stated as
$$
k ! \sim \sqrt{2 \pi k}\left(\frac{k}{e}\right)^{k},
$$
where  $k$ is a positive integer and \textquotedblleft$\sim$\textquotedblright\ can be understood in the sense that the ratio of the two sides tends to 1 as $k$ goes to $\infty$.

\section{Toeplitz Products}\label{section2}
\ \ \ \  In this section, we are going to characterize bounded Toeplitz product $T_{f}T_{g}$ with $f, g\in\mathcal{P}$. For $\alpha\in \mathbb{N}^{n}$ and $z\in \mathbb{C}^{n}$, the functions
$$
e_{\alpha}(z)=\sqrt{\frac{(m+n-1) !(n-1+|\alpha|) !}{\alpha !(n-1) ! (m+n-1+|\alpha|) !}} z^{\alpha}
$$
form an orthonormal basis for $F^{2,m}$, see \cite{ChoZhu2012} for more details.

Given $\alpha=(\alpha_{1},\ldots,\alpha_{n})$, $\beta=(\beta_{1},\ldots,\beta_{n})\in \mathbb{N}^{n}$, the addition and the subtraction of $\alpha$ and $\beta$ are defined by
$$
\alpha \pm \beta:=\left(\alpha_{1} \pm \beta_{1},\ldots, \alpha_{n} \pm \beta_{n}\right).
$$
We call  $\alpha \geq \beta$ (resp. $\alpha>\beta$, $\alpha\leq\beta$, $\alpha<\beta$) if $ \alpha_{i} \geq \beta_{i}$ (resp. $\alpha_{i}>\beta_{i}$, $\alpha_{i}\leq\beta_{i}$, $\alpha_{i}<\beta_{i}$) for each $i=1,\cdots,n$.

We now give a technical result that will be frequently used in the following.
\begin{lemma}\label{ToeL1}
Let $\left\{e_{\alpha}:\alpha\in \mathbb{N}^{n}\right\}$ be any orthonormal basis of $F^{2,m}$. Then for any $\beta$, $\gamma\in \mathbb{N}^{n}$ and $z\in\mathbb{C}^{n}$,
we have
\begin{align*}
T_{z^{\beta}\overline{z}^{\gamma}}e_{\alpha}
=\left\{
   \begin{array}{ll}
     \sqrt{\frac{\alpha !(n-1+|\alpha|) !(n-1+|\alpha+\beta-\gamma|) !}{(\alpha+\beta-\gamma) !(m+n-1+|\alpha|) !(m+n-1+|\alpha+\beta-\gamma|) !}}\frac{(\alpha+\beta)!(m+n-1+|\alpha+\beta|)!}
{\alpha !(n-1+|\alpha+\beta|)! }
e_{\alpha+\beta-\gamma}, \quad& \hbox{$\alpha+\beta-\gamma\geq0$,} \\
     0, & \hbox{\text{otherwise}.}
   \end{array}
 \right.
\end{align*}
\end{lemma}
\begin{proof}
Direct verifications give
\begin{align}\label{(ToeL11)}
T_{z^{\beta}\overline{z}^{\gamma}}e_{\alpha}
&=\sqrt{\frac{(m+n-1) !(n-1+|\alpha|) !}{\alpha !(n-1) ! (m+n-1+|\alpha|) !}}
P_{m}(z^{\alpha+\beta}\overline{z}^{\gamma})
\end{align}
and
\begin{align}\label{(ToeL12)}
P_{m}(z^{\alpha+\beta}\overline{z}^{\gamma})
\nonumber&=\sum_{\eta\in \mathbb{N}^{n}}\left\langle z^{\alpha+\beta}\overline{z}^{\gamma},e_{\eta}\right\rangle_{m}e_{\eta}\\
&=\sum_{\eta\in \mathbb{N}^{n}}\sqrt{\frac{(m+n-1) !(n-1+|\eta|) !}{\eta !(n-1) ! (m+n-1+|\eta|) !}}
\left\langle z^{\alpha+\beta},z^{\eta+\gamma}\right\rangle_{m}e_{\eta}.
\end{align}

For $\eta\neq\alpha+\beta-\gamma$, it is easy to see that
\begin{align}\label{(ToeL18)}
\left\langle z^{\alpha+\beta},z^{\eta+\gamma}\right\rangle_{m}=0.
\end{align}

For $\eta=\alpha+\beta-\gamma$,  applying integration in polar coordinates and using \cite[Lemma 1.11]{Zhu2007}, we obtain
\begin{align*}
\left\langle z^{\alpha+\beta},z^{\eta+\gamma}\right\rangle_{m}
=\frac{(\alpha+\beta)!(n-1)!(m+n-1+|\alpha+\beta|)!}
{(m+n-1)!(n-1+|\alpha+\beta|)!}.
\end{align*}
Notice that if $\alpha+\beta-\gamma\geq0$, then there exists a unique $\eta$ in \eqref{(ToeL12)} such that $\eta=\alpha+\beta-\gamma$. Thus
\begin{align*}
&\hs\hs P_{m}(z^{\alpha+\beta}\overline{z}^{\gamma})\\
&=\sqrt{\frac{(m+n-1) !(n-1+|\alpha+\beta-\gamma|) !}{(\alpha+\beta-\gamma) !(n-1) ! (m+n-1+|\alpha+\beta-\gamma|) !}}
\left\langle z^{\alpha+\beta},z^{\alpha+\beta}\right\rangle_{m}
e_{\alpha+\beta-\gamma}\\
&=\sqrt{\frac{(n-1) !(n-1+|\alpha+\beta-\gamma|) !}{ (\alpha+\beta-\gamma) !( m+n-1) !(m+n-1+|\alpha+\beta-\gamma|) !}}\frac{(\alpha+\beta)!(m+n-1+|\alpha+\beta|)!}
{(n-1+|\alpha+\beta|)! }
e_{\alpha+\beta-\gamma}.
\end{align*}
This together with \eqref{(ToeL11)} gives
\begin{align*}
&T_{z^{\beta}\overline{z}^{\gamma}}e_{\alpha}\\
=&\sqrt{\frac{\alpha !(n-1+|\alpha|) !(n-1+|\alpha+\beta-\gamma|) !}{(\alpha+\beta-\gamma) !(m+n-1+|\alpha|) !(m+n-1+|\alpha+\beta-\gamma|) !}}\frac{(\alpha+\beta)!(m+n-1+|\alpha+\beta|)!}
{\alpha !(n-1+|\alpha+\beta|)! }
e_{\alpha+\beta-\gamma}.
\end{align*}
If $\alpha+\beta-\gamma$ is not larger than 0 and not equal to 0, then
$\eta\neq\alpha+\beta-\gamma$ for all $\eta$ in \eqref{(ToeL12)}, it follows from \eqref{(ToeL11)}-\eqref{(ToeL18)} that
$T_{z^{\beta}\overline{z}^{\gamma}}e_{\alpha}=0.$
This completes the proof.
\end{proof}

In order to state the following lemma effectively,
for any function $f$, we define
\begin{align}
f^{(j)}:=\left\{
                  \begin{array}{ll}
                     f, &\quad \hbox{$j=0$,} \\
                       & \hbox{  } \\
                   \overline{f}, & \quad\hbox{$j=1$.}
                  \end{array}
                \right.
\end{align}

\begin{lemma}\label{ToeL2}
Suppose $\beta=(\beta_{1},\cdots,\beta_{n})$, $\gamma=(\gamma_{1},\cdots,\gamma_{n})$, $k=(k_{1},\cdots,k_{n})$ and
$l=(l_{1},\cdots,l_{n})$ are in $\mathbb{N}^{n}$. For any $z=(z_{1},\cdots,z_{n})\in\mathbb{C}^{n}$, let
\begin{align*}
f_{\beta_{i}}(z_{i})=\sum_{\mu_{i}\leq k_{i}}
a_{\mu_{i}}z_{i}^{\mu_{i}+\beta_{i}}\overline{z}_{i}^{\mu_{i}},\quad
g_{\gamma_{i}}(z_{i})=\sum_{\nu_{i}\leq l_{i}}
b_{\nu_{i}}z_{i}^{\nu_{i}+\gamma_{i}}\overline{z}_{i}^{\nu_{i}},
\end{align*}
where $a_{\mu_{i}}$, $b_{\nu_{i}}$ are constants with $a_{k_{i}}$, $b_{l_{i}}$ nonzero  for each $i=1,\cdots,n$.
For  $i_{1},\cdots,i_{n}$, $j_{1},\cdots,j_{n} \in \left\{0, 1\right\}$, let
\begin{align*}
f_{\beta}(z)=f^{(i_{1})}_{\beta_{1}}(z_{1})\cdots f^{(i_{n})}_{\beta_{n}}(z_{n}),\quad
g_{\gamma}(z)=g^{(j_{1})}_{\gamma_{1}}(z_{1})\cdots g^{(j_{n})}_{\gamma_{n}}(z_{n}).
\end{align*}
Then each of the Toeplitz products
$T_{f_{\beta}}T_{g_{\gamma}}$
is bounded on $F^{2,m}$ if and only if $\beta=\gamma=k = l =(0,\cdots,0)$.
\end{lemma}
\begin{proof}
For simplicity, we set $i=(i_{1},\cdots,i_{n})$, $j=(j_{1},\cdots,j_{n})$ and denote
\begin{align*}
&\theta:=\theta_{\mu,\beta,i}=\left(\mu_{1}+\chi_{\{0\}}(i_{1})\beta_{1},\cdots, \mu_{n}+\chi_{\{0\}}(i_{n})\beta_{n}\right),\\
&\vartheta:=\vartheta_{\mu,\beta,i}=\left(\mu_{1}+\chi_{\{1\}}(i_{1})\beta_{1},\cdots, \mu_{n}+\chi_{\{1\}}(i_{n})\beta_{n}\right),\\
&\varphi:=\varphi_{\nu,\gamma,j}=\left(\nu_{1}+\chi_{\{0\}}(j_{1})\gamma_{1},\cdots, \nu_{n}+\chi_{\{0\}}(j_{n})\gamma_{n}\right),\\
&\psi:=\psi_{\nu,\gamma,j}=\left(\nu_{1}+\chi_{\{1\}}(j_{1})\gamma_{1},\cdots, \nu_{n}+\chi_{\{1\}}(j_{n})\gamma_{n}\right).
\end{align*}
For  $\alpha\in\mathbb{N}^{n}$ satisfying $\alpha_{s}\geq \chi_{\{1\}}(j_{s})\gamma_{s}+\chi_{\{1\}}(i_{s})\beta_{s}$ $(s=1,\cdots,n)$, we apply Lemma \ref{ToeL1} twice to obtain
\begin{align}\label{(ToeL21)}
\nonumber&\hs\hs T_{f_{\beta}}T_{g_{\gamma}}e_{\alpha}\\
\nonumber&=\sum_{\mu_{1}\leq k_{1}}\cdots\sum_{\mu_{n}\leq k_{n}}
\sum_{\nu_{1}\leq l_{1}}\cdots\sum_{\nu_{n}\leq l_{n}}
a^{(i_{1})}_{\mu_{1}}\cdots a^{(i_{n})}_{\mu_{n}}
b^{(j_{1})}_{\nu_{1}}\cdots b^{(j_{n})}_{\nu_{n}}
T_{z^{\theta}\bar{z}^{\vartheta}}
T_{z^{\varphi}\bar{z}^{\psi}}e_{\alpha}\\
\nonumber&=\sum_{\mu_{1}\leq k_{1}}\cdots\sum_{\mu_{n}\leq k_{n}}
\sum_{\nu_{1}\leq l_{1}}\cdots\sum_{\nu_{n}\leq l_{n}}
a^{(i_{1})}_{\mu_{1}}\cdots a^{(i_{n})}_{\mu_{n}}
b^{(j_{1})}_{\nu_{1}}\cdots b^{(j_{n})}_{\nu_{n}}
T_{z^{\theta}\bar{z}^{\vartheta}}\\
\nonumber&\hs\hs\times\Bigg(\sqrt{\frac{\alpha !(n-1+|\alpha|) !(n-1+|\alpha+\varphi-\psi|) !}{(\alpha+\varphi-\psi) !(m+n-1+|\alpha|) !(m+n-1+|\alpha+\varphi-\psi|)!}}\\
\nonumber&\hs\hs\times\frac{(\alpha+\varphi)!(m+n-1+|\alpha+\varphi|)!}
{\alpha !(n-1+|\alpha+\varphi|)!}\Bigg)
e_{\alpha+\varphi-\psi}\\
\nonumber&=\sum_{\mu_{1}\leq k_{1}}\cdots\sum_{\mu_{n}\leq k_{n}}
\sum_{\nu_{1}\leq l_{1}}\cdots\sum_{\nu_{n}\leq l_{n}}
a^{(i_{1})}_{\mu_{1}}\cdots a^{(i_{n})}_{\mu_{n}}
b^{(j_{1})}_{\nu_{1}}\cdots b^{(j_{n})}_{\nu_{n}}\\
\nonumber&\hs\hs\times\Bigg(\sqrt{\frac{\alpha !(n-1+|\alpha|) !(n-1+|\alpha+\varphi-\psi|) !}{(\alpha+\varphi-\psi) !(m+n-1+|\alpha|) !(m+n-1+|\alpha+\varphi-\psi|)!}}\frac{(\alpha+\varphi)!
(m+n-1+|\alpha+\varphi|)!}
{\alpha !(n-1+|\alpha+\varphi|)!}\\
\nonumber&\hs\hs\times
\sqrt{\frac{(\alpha+\varphi-\psi)!(n-1+|\alpha+\varphi-\psi|)!
(n-1+|\alpha+\varphi-\psi+\theta-\vartheta|)!}
{(\alpha+\varphi-\psi+\theta-\vartheta)!(m+n-1+|\alpha+\varphi-\psi|)!
(m+n-1+|\alpha+\varphi-\psi+\theta-\vartheta|)!}}\\
\nonumber&\hs\hs\times\frac{(\alpha+\varphi-\psi+\theta)!
(m+n-1+|\alpha+\varphi-\psi+\theta|)!}
{(\alpha+\varphi-\psi)!
(n-1+|\alpha+\varphi-\psi+\theta|)!}\Bigg)
e_{\alpha+\varphi-\psi+\theta-\vartheta}\\
&=\sum_{\mu_{1}\leq k_{1}}\cdots\sum_{\mu_{n}\leq k_{n}}
\sum_{\nu_{1}\leq l_{1}}\cdots\sum_{\nu_{n}\leq l_{n}}
a^{(i_{1})}_{\mu_{1}}\cdots a^{(i_{n})}_{\mu_{n}}
b^{(j_{1})}_{\nu_{1}}\cdots b^{(j_{n})}_{\nu_{n}}
A^{\theta\vartheta\varphi\psi}_{\alpha}
e_{\alpha+\varphi-\psi+\theta-\vartheta},
\end{align}
where
\begin{align*}
A^{\theta\vartheta\varphi\psi}_{\alpha}
&:=\sqrt{\frac{\alpha !(n-1+|\alpha|)!(n-1+|\alpha+\varphi-\psi+\theta-\vartheta|)!}
{(\alpha+\varphi-\psi+\theta-\vartheta)!(m+n-1+|\alpha|)!
(m+n-1+|\alpha+\varphi-\psi+\theta-\vartheta|)!}}\\
&\hs\hs\times\frac{(\alpha+\varphi)!(\alpha+\varphi-\psi+\theta)!
(m+n-1+|\alpha+\varphi|)!(n-1+|\alpha+\varphi-\psi|)!}
{\alpha !(\alpha+\varphi-\psi)!
(n-1+|\alpha+\varphi|)!(m+n-1+|\alpha+\varphi-\psi|)!}\\
&\hs\hs\times\frac{(m+n-1+|\alpha+\varphi-\psi+\theta|)!}
{(n-1+|\alpha+\varphi-\psi+\theta|)!}.
\end{align*}
An application of Stirling's formula implies that
\begin{align}\label{(ToeL22)}
A^{\theta\vartheta\varphi\psi}_{\alpha}
\sim \alpha^{\frac{1}{2}(\varphi+\psi+\theta+\vartheta)}
=\alpha^{\frac{1}{2}(\beta+\gamma)+\mu+\nu},
\quad\mu\leq k,\ \nu\leq l.
\end{align}
Since $a_{k_{i}}$, $b_{l_{i}}$ are nonzero constants for each $i=1,\cdots,n$, it follows from \eqref{(ToeL21)} and \eqref{(ToeL22)} that
\begin{align*}
\left\|T_{f_{\beta}}T_{g_{\gamma}}
e_{\alpha}\right\|_{2,m}
&=\left|\sum_{\mu_{1}\leq k_{1}}\cdots\sum_{\mu_{n}\leq k_{n}}
\sum_{\nu_{1}\leq l_{1}}\cdots\sum_{\nu_{n}\leq l_{n}}
a^{(i_{1})}_{\mu_{1}}\cdots a^{(i_{n})}_{\mu_{n}}
b^{(j_{1})}_{\nu_{1}}\cdots b^{(j_{n})}_{\nu_{n}}
A^{\theta\vartheta\varphi\psi}_{\alpha}\right|\\
&\sim \left|a^{(i_{1})}_{k_{1}}\cdots a^{(i_{n})}_{k_{n}}
b^{(j_{1})}_{l_{1}}\cdots b^{(j_{n})}_{l_{n}}
\alpha^{\frac{1}{2}(\beta+\gamma)+k+l}\right|.
\end{align*}
Therefore, if we denote
$$\mathcal{A}=\{\alpha\in \mathbb{N}^{n}:\alpha_{s}\geq \chi_{\{1\}}(j_{s})\gamma_{s}+\chi_{\{1\}}(i_{s})\beta_{s}\ \text{for any}\ s=1,\cdots,n\},$$
then the Toeplitz product $T_{f_{\beta}}T_{g_{\gamma}}$ is bounded
if and only if
$$\left\{ \left\|T_{f_{\beta}}T_{g_{\gamma}}e_{\alpha}
\right\|_{2,m}\right\}_{\alpha\in\mathcal{A}}
$$
is bounded on $F^{2,m}$, which is equivalent to $\beta=\gamma=k = l =(0,\cdots,0)$.
This completes the proof of Lemma \ref{ToeL2}.
\end{proof}

Next, we will use Lemma \ref{ToeL2} to prove the main theorem in this section. To this end, we first
recall that, if $f$ is a polynomial in  $z,\overline{z}\in\mathbb{C}^{n}$, then there exist $k=(k_{1},\cdots,k_{n})$ and
$l=(l_{1},\cdots,l_{n})\in\mathbb{N}^{n}$ such that
\begin{align}\label{fzz}
f(z,\overline{z})=\prod_{s=1}^{n}\left(\sum_{\beta_{s}\leq k_{s}}\sum_{\gamma_{s}\leq l_{s}}a_{\beta_{s}\gamma_{s},s}z_{s}^{\beta_{s}}\overline{z}_{s}^{\gamma_{s}}\right).
\end{align}

For any $s=1,\cdots,n$, let
\begin{align*}
i_{0,s}=\min\left\{\beta_{s}-\gamma_{s}:a_{\beta_{s}\gamma_{s}}\neq0,\beta_{s}\leq k_{s}, \gamma_{s}\leq l_{s}\right\},
\end{align*}
and
\begin{align*}
i_{1,s}=\max\left\{\beta_{s}-\gamma_{s}:a_{\beta_{s}\gamma_{s}}\neq0,\beta_{s}\leq k_{s}, \gamma_{s}\leq l_{s}\right\}.
\end{align*}

For each integer $\theta_{s}$ satisfying  $i_{0,s}\leq \theta_{s} \leq i_{1,s}$ $(s=1,\cdots,n)$, let $F_{\theta_{s}}(z_{s}, \overline{z_{s}})$ be the sum of all those terms $a_{\beta_{s}\gamma_{s}}z^{\beta_{s}}\overline{z}^{\gamma_{s}}$ in the polynomial formula $(\ref{fzz})$ of $f$ such that $\beta_{s}-\gamma_{s}= \theta_{s}$. If there is no such kind of term, we set $F_{\theta_{s}}=0$. Then $F_{\theta_{s}}$ is of the same form as the function $f_{\beta_{s}}$ (if $\theta_{s} \geq 0$) or the complex conjugate of  $f_{\beta_{s}}$ (if $\theta_{s} < 0$) in Lemma \ref{ToeL2}.
Thus, with this new notation, the expression in $(\ref{fzz})$ may be rewritten as
\begin{align*}
f(z,\overline{z})=\prod_{s=1}^{n}\left(\sum_{\theta_{s}=i_{0,s}}^{i_{1,s}}
F_{\theta_{s}}(z_{s}, \overline{z_{s}})\right).%
\end{align*}

%
%

Now, we give the proof the first main result.

\begin{proof}[\pmb{Proof of Theorem \ref{TP}}]
If both $f$ and $g$ are constants, then it is easy to check  Toeplitz operators  $T_{f}$ and $T_{g}$ are both bounded on $F^{2,m}$. Hence the Toeplitz product $T_{f}T_{g}$ is bounded on $F^{2,m}$.

Conversely, suppose the Toeplitz product $T_{f}T_{g}$ is bounded. Since
both $f$ and $g$ are polynomials in $z,\overline{z}\in\mathbb{C}^{n}$, from the above discussion, $f$ and $g$ admit expansions:
\begin{align*}
f(z,\overline{z})=\prod_{s=1}^{n}\left(\sum_{\theta_{s}=i_{0,s}}^{i_{1,s}}
F_{\theta_{s}}(z_{s}, \overline{z_{s}})\right),\quad%
g(z,\overline{z})=\prod_{t=1}^{n}\left(\sum_{\tau_{t}=j_{0,t}}^{j_{1,t}}
G_{\tau_{t}}(z_{t}, \overline{z_{t}})\right),%
\end{align*}
where $F_{i_{0,s}}(z_{s}, \overline{z_{s}})$, $F_{i_{1,s}}(z_{s}, \overline{z_{s}})$, $G_{j_{0,t}}(z_{t}, \overline{z_{t}})$ and $G_{j_{1,t}}(z_{t}, \overline{z_{t}})$ are nonzero for all $s,t=1,\cdots,n$. In what follows, we write
$$F_{\theta_{s}}:=F_{\theta_{s}}(z_{s}, \overline{z_{s}}), \quad G_{\tau_{t}}:=G_{\tau_{t}}(z_{t}, \overline{z_{t}})$$
for simplicity. Therefore
\begin{align}
T_{f}T_{g}e_{\alpha}
\nonumber&=\sum_{\theta_{1}=i_{0,1}}^{i_{1,1}}\cdots
\sum_{\theta_{n}=i_{0,n}}^{i_{1,n}}
\sum_{\tau_{1}=j_{0,1}}^{j_{1,1}}\cdots\sum_{\tau_{n}=j_{0,n}}^{j_{1,n}}
T_{F_{\theta_{1}}\cdots F_{\theta_{n}}}
T_{G_{\tau_{1}}\cdots G_{\tau_{n}}}e_{\alpha}\\
\label{(TP1)}&=T_{F_{i_{1,1}}\cdots F_{i_{1,n}}}
T_{G_{j_{1,1}}\cdots G_{j_{1,n}}}e_{\alpha}
+\sum_{\substack{(\theta_{1},\cdots,\theta_{n},\tau_{1},\cdots,\tau_{n})
\neq\\(i_{1,1},\cdots,i_{1,n},j_{1,1},\cdots,j_{1,n})}}
T_{F_{\theta_{1}}\cdots F_{\theta_{n}}}
T_{G_{\tau_{1}}\cdots G_{\tau_{n}}}e_{\alpha}.
\end{align}
Set multi-index
$$
\kappa=\left(\max \left\{\left|i_{0,1}\right|,\left|i_{1,1}\right|\right\}+\max \left\{\left|j_{0,1}\right|,\left|j_{1,1}\right|\right\}
,\cdots,\max \left\{\left|i_{0,n}\right|,\left|i_{1,n}\right|\right\}+\max \left\{\left|j_{0,n}\right|,\left|j_{1,n}\right|\right\}\right).
$$
It follows from the definitions of $F_{\theta_{s}}$, $G_{\tau_{t}}$ and the proof of Lemma \ref{ToeL2} that for any $\alpha\geq\kappa$, $\beta=\left(\theta_{1},\cdots,\theta_{n}\right),\
\gamma=\left(\tau_{1},\cdots,\tau_{n}\right)$ with $i_{0,s} \leq \theta_{s} \leq i_{1,s}$ and $j_{0,t} \leq \tau_{t}\leq j_{1,t}$ ($s,t=1,\cdots,n$) such that $(\theta_{1},\cdots,\theta_{n},\tau_{1},\cdots,\tau_{n})
\neq(i_{1,1},\cdots,i_{1,n},j_{1,1},\cdots,j_{1,n})$, we have
$$
T_{F_{\theta_{1}}\cdots F_{\theta_{n}}}
T_{G_{\tau_{1}}\cdots G_{\tau_{n}}}e_{\alpha}\in \text{Span}\{e_{\alpha+\beta+\gamma}\}.
$$
Notice that the first term of (\ref{(TP1)})
$$T_{F_{i_{1,1}}\cdots F_{i_{1,n}}}
T_{G_{j_{1,1}}\cdots G_{j_{1,n}}}e_{\alpha}\in \text{Span}\{e_{\alpha+\beta^{\prime}+\gamma^{\prime}}\},$$
where $\beta^{\prime}=(i_{1,1},\cdots,i_{1,n})$ and $\gamma^{\prime}=(j_{1,1},\cdots,j_{1,n})$, we see that  $T_{F_{i_{1,1}}\cdots F_{i_{1,n}}}T_{G_{j_{1,1}}\cdots G_{j_{1,n}}}e_{\alpha}$ is orthogonal to the second term of (\ref{(TP1)}) for $\alpha\geq\kappa$.
It follows that
\begin{align*}
\left\|T_{f}T_{g}e_{\alpha}\right\|_{2,m}
\geq\left\|T_{F_{i_{1,1}}\cdots F_{i_{1,n}}}
T_{G_{j_{1,1}}\cdots G_{j_{1,n}}}e_{\alpha}\right\|_{2,m}.
\end{align*}
By the proof of Lemma \ref{ToeL2}, it is easy to see that $T_{F_{i_{1,1}}\cdots F_{i_{1,n}}}
T_{G_{j_{1,1}}\cdots G_{j_{1,n}}}$ is bounded
if and only if the sequence
$$\left\{\left\|T_{F_{i_{1,1}}\cdots F_{i_{1,n}}}
T_{G_{j_{1,1}}\cdots G_{j_{1,n}}}e_{\alpha}\right\|_{2,m}\right\}_{\alpha \geq \kappa}$$
is bounded.
Thus the boundedness of $T_{f}T_{g}$ implies  the boundedness of $T_{F_{i_{1,1}}\cdots F_{i_{1,n}}}T_{G_{j_{1,1}}\cdots G_{j_{1,n}}}$. This along with Lemma \ref{ToeL2} implies that $F_{i_{1,1}},\cdots, F_{i_{1,n}}$ and $G_{j_{1,1}},\cdots, G_{j_{1,n}}$ must be constants.
Similarly, we can also conclude that $T_{F_{i_{0,1}}\cdots F_{0_{1,n}}}
T_{G_{j_{0,1}}\cdots G_{j_{0,n}}}$  is bounded if (\ref{(TP1)}) is replaced by
\begin{align}
T_{f}T_{g}e_{\alpha}
\label{(TP3)}
=T_{F_{i_{0,1}}\cdots F_{i_{0,n}}}
T_{G_{j_{0,1}}\cdots G_{j_{0,n}}}e_{\alpha}
+\sum_{\substack{(\theta_{1},\cdots,\theta_{n},\tau_{1},\cdots,\tau_{n})
\neq\\(i_{0,1},\cdots,i_{0,n},j_{0,1},\cdots,j_{0,n})}}
T_{F_{\theta_{1}}\cdots F_{\theta_{n}}}
T_{G_{\tau_{1}}\cdots G_{\tau_{n}}}e_{\alpha},
\end{align}
where the summation  is taken over all $i_{0,1}\leq\theta_{1}\leq i_{1,1},\cdots,i_{0,n}\leq\theta_{n}\leq i_{1,n}$,
$j_{0,1}\leq\tau_{1}\leq j_{1,1},\cdots,j_{0,n}\leq\tau_{n}\leq j_{1,n}$ such that $(\theta_{1},\cdots,\theta_{n},\tau_{1},\cdots,\tau_{n})
\neq(i_{0,1},\cdots,i_{0,n},j_{0,1},\cdots,j_{0,n})$.
By Lemma \ref{ToeL2} again, $F_{i_{0,1}},\cdots, F_{i_{0,n}}$ and $G_{j_{0,1}},\cdots, G_{j_{0,n}}$ must be constants. Thus $f$ and $g$ are constants.
This completes the proof of the theorem.
\end{proof}
\begin{corollary}
Assume that $f$ is a polynomial in $z,\overline{z}\in\mathbb{C}^{n}$. Then the Toeplitz operator $T_{f}$ is bounded on $F^{2,m}$ if and only if $f$ is a constant.
\end{corollary}
\begin{proof}
It follows from Theorem \ref{TP} by setting $g = 1$ or $g = f$.
\end{proof}
\section{Hankel Products}\label{section3}
\ \ \ \  In this section, we are  to characterize bounded Hankel Products $H^{*}_{f}H_{g}$ with $f, g\in\mathcal{P}$. For technical reasons, we require the following lemma.
\begin{lemma}\label{HanL1}
Assume $\beta=(\beta_{1},\cdots,\beta_{n})$, $\gamma=(\gamma_{1},\cdots,\gamma_{n})$, $\mu=(\mu_{1},\cdots,\mu_{n})$ and $\nu=(\nu_{1},\cdots,\nu_{n})$ are all in $\mathbb{N}^{n}$. Let
$f = z^{\beta}\overline{z}^{\gamma}$ and $g = z^{\mu}\overline{z}^{\nu}$ for $z,\overline{z}\in\mathbb{C}^{n}$. Then for any
$\alpha\geq (|\gamma_{1} - \beta_{1}| + |\mu_{1} - \nu_{1}|,\cdots,|\gamma_{n} - \beta_{n}| + |\mu_{n} - \nu_{n}|)$, we have
$$
H_{f}^{*} H_{g} e_{\alpha}=A_{\alpha} e_{\alpha+\gamma +\mu- \beta - \nu},
$$
where
\begin{align}\label{(HanL11)}
\nonumber A_{\alpha}&=\bigg(\frac{(\alpha+\gamma+\mu)!(m+n-1+|\alpha+\gamma+\mu|)!}
{\alpha !(n-1+|\alpha+\gamma+\mu|)! }-\frac{(\alpha+\mu)!(\alpha+\gamma+\mu-\nu)!}
{\alpha !(\alpha+\mu- \nu) ! }\\
\nonumber&\hspace{0.6cm}\times\frac{(m+n-1+|\alpha+\mu|)!(n-1+|\alpha +\mu - \nu|) !(m+n-1+|\alpha+\gamma+\mu-\nu|)!}
{(n-1+|\alpha+\mu|)!(m+n-1+|\alpha +\mu - \nu|) !(n-1+|\alpha+\gamma+\mu-\nu|)! }\bigg)\\
&\hspace{0.6cm}\times\sqrt{\frac{\alpha !(n-1+|\alpha|) !(n-1+|\alpha+\gamma +\mu- \beta - \nu|) !}{(\alpha+\gamma +\mu- \beta - \nu) !(m+n-1+|\alpha|) !(m+n-1+|\alpha+\gamma +\mu- \beta - \nu|) !}}.
\end{align}
Furthermore, $A_{\alpha}=0$ if and only if $\gamma= 0$ or $\nu = 0$. And if $A_{\alpha}\neq0$, then
\begin{align}\label{(HanL101)}
A_{\alpha}\sim \left(\sum_{j=1}^{n}\gamma_{j}\nu_{j}\alpha_{j}^{-1}\right)
\alpha^{\frac{\beta+\nu+\gamma+\mu}{2}}.
\end{align}
\end{lemma}
\begin{proof}
We only give the proof for the case of $m\neq0$, since the case of $m=0$ is much simpler.
It is easy to verify that
\begin{align}\label{(HanL12)}
H_{f}^{*} H_{g}=T_{\bar{f} g}-T_{\bar{f}} T_{g}=T_{z^{\gamma+\mu} \bar{z}^{\beta+\nu}}-T_{z^{\gamma} \bar{z}^{\beta}} T_{z^{\mu} \bar{z}^{\nu}}.
\end{align}
It follows from Lemma \ref{ToeL1} that for any $\alpha\geq (|\gamma_{1} - \beta_{1}| + |\mu_{1} - \nu_{1}|,\cdots,|\gamma_{n} - \beta_{n}| + |\mu_{n} - \nu_{n}|)$, we have
\begin{align}\label{(HanL13)}
\nonumber&\hspace{0.6cm}T_{z^{\gamma+\mu} \bar{z}^{\beta+\nu}}e_{\alpha}\\
\nonumber&=\sqrt{\frac{\alpha !(n-1+|\alpha|) !(n-1+|\alpha+\gamma +\mu- \beta - \nu|) !}{(\alpha+\gamma +\mu- \beta - \nu) !(m+n-1+|\alpha|) !(m+n-1+|\alpha+\gamma +\mu- \beta - \nu|) !}}\\
&\hspace{0.6cm}\times\frac{(\alpha+\gamma+\mu)!(m+n-1+|\alpha+\gamma+\mu|)!}
{\alpha !(n-1+|\alpha+\gamma+\mu|)! }
e_{\alpha+\gamma +\mu- \beta - \nu}.
\end{align}
Applying Lemma \ref{ToeL1} again, we obtain
\begin{align}\label{(HanL14)}
\nonumber&\hspace{0.6cm}T_{z^{\gamma} \bar{z}^{\beta}} T_{z^{\mu}\bar{z}^{\nu}}e_{\alpha}\\
\nonumber&=T_{z^{\gamma} \bar{z}^{\beta}} \sqrt{\frac{\alpha !(n-1+|\alpha|) !(n-1+|\alpha +\mu - \nu|) !}{(\alpha+\mu- \nu)!(m+n-1+|\alpha|) !(m+n-1+|\alpha +\mu- \nu|) !}}\\
\nonumber&\hspace{0.6cm}\times\frac{(\alpha+\mu)!(m+n-1+|\alpha+\mu|)!}
{\alpha !(n-1+|\alpha+\mu|)! }
e_{\alpha +\mu- \nu}\\
\nonumber&=\sqrt{\frac{(\alpha+\mu- \nu) !(n-1+|\alpha+\mu- \nu|) !(n-1+|\alpha+\gamma +\mu- \beta - \nu|) !}{(\alpha+\gamma +\mu- \beta - \nu) !(m+n-1+|\alpha+\mu- \nu|) !(m+n-1+|\alpha+\gamma +\mu- \beta - \nu|) !}}\\
\nonumber&\hspace{0.6cm}\times\frac{(\alpha+\mu)!(\alpha+\gamma+\mu-\nu)!(m+n-1+|\alpha+\mu|)!(m+n-1+|\alpha+\gamma+\mu-\nu|)!}
{\alpha !(\alpha+\mu- \nu) !(n-1+|\alpha+\mu|)!(n-1+|\alpha+\gamma+\mu-\nu|)! }\\
\nonumber&=\sqrt{\frac{\alpha !(n-1+|\alpha|)!(n-1+|\alpha+\gamma +\mu- \beta - \nu|) !}{(\alpha+\gamma +\mu- \beta - \nu) !(m+n-1+|\alpha|) !(m+n-1+|\alpha+\gamma +\mu- \beta - \nu|) !}}\\
\nonumber&\hspace{0.6cm}\times\frac{(m+n-1+|\alpha+\mu|)!(n-1+|\alpha +\mu - \nu|) !(m+n-1+|\alpha+\gamma+\mu-\nu|)!}
{(n-1+|\alpha+\mu|)!(m+n-1+|\alpha +\mu - \nu|) !(n-1+|\alpha+\gamma+\mu-\nu|)! }\\
&\hspace{0.6cm}\times\frac{(\alpha+\mu)!(\alpha+\gamma+\mu-\nu)!}
{\alpha !(\alpha+\mu- \nu) ! }e_{\alpha+\gamma +\mu- \beta - \nu}.
\end{align}
Combining \eqref{(HanL12)}-\eqref{(HanL14)}, we deduce the explicit formula for $A_{\alpha}$ in \eqref{(HanL11)}. From  this formula, it is not hard to see that $A_{\alpha}=0$ is equivalent to  $\gamma= 0$ or $\nu = 0$.

If $A_{\alpha}\neq0$, then by Stirling's formula, we have
\begin{align}\label{(HanL15)}
\sqrt{\frac{\alpha !(n-1+|\alpha|) !(n-1+|\alpha+\gamma +\mu- \beta - \nu|) !}{(\alpha+\gamma +\mu- \beta - \nu) !(n-1+m+|\alpha|)!(n-1+m+|\alpha+\gamma +\mu- \beta - \nu|) !}}
\sim  \alpha^{\frac{\beta+\nu-\gamma-\mu}{2}}|\alpha|^{-m}.
\end{align}
Denote
\begin{align}\label{(HanL16)}
\nonumber B_{\alpha}:&=\frac{(\alpha+\gamma+\mu)!(m+n-1+|\alpha+\gamma+\mu|)!}
{\alpha !(n-1+|\alpha+\gamma+\mu|)! }-\frac{(\alpha+\mu)!(\alpha+\gamma+\mu-\nu)!}
{\alpha !(\alpha+\mu- \nu) ! }\\
&\hspace{0.6cm}\times\frac{(m+n-1+|\alpha+\mu|)!(n-1+|\alpha +\mu - \nu|) !(m+n-1+|\alpha+\gamma+\mu-\nu|)!}
{(n-1+|\alpha+\mu|)!(m+n-1+|\alpha +\mu - \nu|) !(n-1+|\alpha+\gamma+\mu-\nu|)! }
\end{align}
for simplicity.
Next, we study the asymptotic behavior of $B_{\alpha}$ as each component $\alpha_j$ tends to infinity. Firstly, we estimate the first term of $B_{\alpha}$.
\begin{align}\label{(HanL17)}
\nonumber&\hspace{0.6cm}\frac{(\alpha+\gamma+\mu)!(m+n-1+|\alpha+\gamma+\mu|)!}
{ \alpha !(n-1+|\alpha+\gamma+\mu|)!}\\
\nonumber&=\left(\prod_{j=1}^{n}\prod_{i=1}^{\gamma_{j}+\mu_{j}}
(\alpha_{j}+i)\right)
 \prod_{i=1}^{m}(n-1+|\alpha+\gamma+\mu|+i)\\
\nonumber&=\left(\prod_{j=1}^{n}
\Bigg(\alpha_{j}^{\gamma_{j}+\mu_{j}}
+\bigg(\sum_{i=1}^{\gamma_{j}+\mu_{j}}i\bigg)
\alpha_{j}^{\gamma_{j}+\mu_{j}-1}
+
O(\alpha_{j}^{\gamma_{j}+\mu_{j}-2})\Bigg)\right)\\
\nonumber&\hspace{0.6cm}\times\left(|\alpha|^{m}
+\bigg(\sum_{i=1}^{m}(n-1+|\gamma|+|\mu|+i)
\bigg)|\alpha|^{m-1}+O(|\alpha|^{m-2})\right)\\
\nonumber&=\left(\alpha^{\gamma+\mu}+\sum_{j=1}^{n}
\bigg(\sum_{i=1}^{\gamma_{j}+\mu_{j}}i
\bigg)\alpha_{j}^{-1}\alpha^{\gamma+\mu}
+\sum_{j,k=1}^{n}O(\alpha^{\gamma+\mu}\alpha_{j}^{-1}\alpha_{k}^{-1})\right)\\
\nonumber&\hspace{0.6cm}\times
\left(|\alpha|^{m}+\bigg(\sum_{i=1}^{m}(n-1+|\gamma|+|\mu|+i)
\bigg)|\alpha|^{m-1}+O(|\alpha|^{m-2})\right)\\
\nonumber&=\alpha^{\gamma+\mu}|\alpha|^{m}+\sum_{j=1}^{n}
\left(\sum_{i=1}^{\gamma_{j}+\mu_{j}}i\right)\alpha_{j}^{-1}
\alpha^{\gamma+\mu}|\alpha|^{m}
+\left(\sum_{i=1}^{m}(n-1+|\gamma|
+|\mu|+i)\right)
\alpha^{\gamma+\mu}|\alpha|^{m-1}\\
&\hspace{0.6cm}+O(\alpha^{\gamma+\mu}
|\alpha|^{m-2})+\sum_{j=1}^{n}O(
\alpha^{\gamma+\mu}\alpha_{j}^{-1}|\alpha|^{m-1})
+\sum_{j,k=1}^{n}O(\alpha^{\gamma+\mu}
\alpha_{j}^{-1}\alpha_{k}^{-1}|\alpha|^{m}).
\end{align}
Besides,
\begin{align}\label{(HanL18)}
&\hspace{0.6cm}\frac{(\alpha+\mu)!(\alpha+\gamma+\mu-\nu)!}
{\alpha !(\alpha+\mu- \nu) ! }\nonumber\\
&=\left(\prod_{j=1}^{n}\prod_{i=1}^{\mu_{j}}(\alpha_{j}+i)\right)\left(\prod_{j=1}^{n}\prod_{i=1}^{\gamma_{j}}(\alpha_{j}+\mu_{j}-\nu_{j}+i)\right)\nonumber\\
&=\left(\prod_{j=1}^{n}\left(\alpha_{j}^{\mu_{j}}+\left(\sum_{i=1}^{\mu_{j}}i\right)
\alpha_{j}^{\mu_{j}-1}+O(\alpha_{j}^{\mu_{j}-2})\right)\right)\nonumber\\
&\hspace{0.6cm}\times\left(\prod_{j=1}^{n}\left(\alpha_{j}^{\gamma_{j}}+\left(\sum_{i=1}^{\gamma_{j}}
(\mu_{j}- \nu_{j}+i)\right)\alpha_{j}^{\gamma_{j}-1}
+O(\alpha_{j}^{\gamma_{j}-2})\right)\right)\nonumber\\
&=\left(\alpha^{\mu}
+\sum_{j=1}^{n}\left(\sum_{i=1}^{\mu_{j}}i\right)
\alpha_{j}^{-1}\alpha^{\mu}+
\sum_{j,k=1}^{n}O(\alpha^{\mu}\alpha_{j}^{-1}\alpha_{k}^{-1})\right)\nonumber\\
&\hspace{0.6cm}\times\left(\alpha^{\gamma}+\sum_{j=1}^{n}\left(\sum_{i=1}^{\gamma_{j}}
(\mu_{j}- \nu_{j}+i)\right)\alpha_{j}^{-1}\alpha^{\gamma}
+\sum_{j,k=1}^{n}O(\alpha^{\gamma}\alpha_{j}^{-1}\alpha_{k}^{-1})\right)\nonumber\\
&=\alpha^{\mu+\gamma}+\sum_{j=1}^{n}
\left(\left(\sum_{i=1}^{\mu_{j}}i\right)+\left(\sum_{i=1}^{\gamma_{j}}
(\mu_{j}- \nu_{j}+i)\right)\right)\alpha_{j}^{-1}\alpha^{\mu+\gamma}
+\sum_{j,k=1}^{n}O(\alpha^{\mu+\gamma}\alpha_{j}^{-1}\alpha_{k}^{-1})
\end{align}
and
\begin{align}\label{(HanL19)}
\nonumber&\frac{(m+n-1+|\alpha+\mu|)!(n-1+|\alpha +\mu - \nu|) !(m+n-1+|\alpha+\gamma+\mu-\nu|)!}
{(n-1+|\alpha+\mu|)!(m+n-1+|\alpha +\mu - \nu|) !(n-1+|\alpha+\gamma+\mu-\nu|)! }\\
\nonumber=&\prod_{i=1}^{m}
\frac{(n-1+|\alpha+\mu|+i)(n-1+|\alpha+\gamma+\mu-\nu|+i)}
{(n-1+|\alpha+\mu-\nu|+i)}\\
\nonumber=&\frac{|\alpha|^{m}
+\left(\sum_{i=1}^{m}(n-1+|\mu|+i)
\right)|\alpha|^{m-1}+O(|\alpha|^{m-2})}
{|\alpha|^{m}+\left(\sum_{i=1}^{m}(n-1+|\mu|-|\nu|+i)
\right)|\alpha|^{m-1}+O(|\alpha|^{m-2})}\\
\nonumber\hspace{0.6cm}&\times\left(|\alpha|^{m}
+\left(\sum_{i=1}^{m}(n-1+|\mu|+|\gamma|-|\nu|+i)
\right)|\alpha|^{m-1}+O(|\alpha|^{m-2})\right)\\
=&|\alpha|^{m}+\left(\sum_{i=1}^{m}(n-1+|\mu|+|\gamma|+i)
\right)|\alpha|^{m-1}+O(|\alpha|^{m-2}),
\end{align}
which implies that
\begin{align}\label{(HanL20)}
&\frac{(\alpha+\mu)!(\alpha+\gamma+\mu-\nu)!}
{\alpha !(\alpha+\mu- \nu) ! }\times\frac{(m+n-1+|\alpha+\mu|)!(n-1+|\alpha +\mu - \nu|) !(m+n-1+|\alpha+\gamma+\mu-\nu|)!}
{(n-1+|\alpha+\mu|)!(m+n-1+|\alpha +\mu - \nu|) !(n-1+|\alpha+\gamma+\mu-\nu|)! }\nonumber\\
=&\alpha^{\gamma+\mu}|\alpha|^{m}
+\sum_{j=1}^{n}\left(\left(\sum_{i=1}^{\mu_{j}}i\right)+
\left(\sum_{i=1}^{\gamma_{j}}(\mu_{j}- \nu_{j}+i)\right)\right)\alpha_{j}^{-1}\alpha^{\gamma+\mu}|\alpha|^{m}\nonumber\\
\hspace{0.6cm}&+\left(\sum_{i=1}^{m}(n-1+|\mu|+|\gamma|+i)
\right)\alpha^{\gamma+\mu}|\alpha|^{m-1}\nonumber\\
\hspace{0.6cm}&+O(\alpha^{\gamma+\mu}|\alpha|^{m-2})
+\sum_{j=1}^{n}O(\alpha^{\gamma+\mu}
\alpha_{j}^{-1}|\alpha|^{m-1})+\sum_{j,k=1}^{n}
O(\alpha^{\gamma+\mu}\alpha_{j}^{-1}\alpha_{k}^{-1}|\alpha|^{m}).
\end{align}
Subtracting \eqref{(HanL20)} from \eqref{(HanL17)}, we obtain
\begin{align*}
B_{\alpha}=\left(\sum_{j=1}^{n}\gamma_{j}\nu_{j}\alpha_{j}^{-1}\right)
\alpha^{\gamma+\mu}|\alpha|^{m}
+O(\alpha^{\gamma+\mu}|\alpha|^{m-2})
+\sum_{j=1}^{n}O(\alpha^{\gamma+\mu}\alpha_{j}^{-1}|\alpha|^{m-1})
+\sum_{j,k=1}^{n}O(\alpha^{\gamma+\mu}
\alpha_{j}^{-1}\alpha_{k}^{-1}|\alpha|^{m}).
\end{align*}
This along with \eqref{(HanL15)} gives
\begin{align*}
A_{\alpha}\sim \left(\sum_{j=1}^{n}\gamma_{j}\nu_{j}\alpha_{j}^{-1}\right)
\alpha^{\frac{\beta+\nu+\gamma+\mu}{2}}.
\end{align*}
This completes the proof of Lemma \ref{HanL1}.
\end{proof}

\begin{lemma}\label{HanL2}
Suppose $\beta=(\beta_{1},\cdots,\beta_{n})$, $\gamma=(\gamma_{1},\cdots,\gamma_{n})$, $k=(k_{1},\cdots,k_{n})$ and
$l=(l_{1},\cdots,l_{n})$ are in $\mathbb{N}^{n}$. For any $z=(z_{1},\cdots,z_{n})\in\mathbb{C}^{n}$, let
\begin{align*}
f_{\beta_{i}}(z_{i})=\sum_{\mu_{i}\leq k_{i}}
a_{\mu_{i}}z_{i}^{\mu_{i}+\beta_{i}}\overline{z}_{i}^{\mu_{i}},\quad
g_{\gamma_{i}}(z_{i})=\sum_{\nu_{i}\leq l_{i}}
b_{\nu_{i}}z_{i}^{\nu_{i}+\gamma_{i}}\overline{z}_{i}^{\nu_{i}},
\end{align*}
where $a_{\mu_{i}}$, $b_{\nu_{i}}$ are constants with $a_{k_{i}}$, $b_{l_{i}}$ nonzero  for each $i=1,\cdots,n$. For $i_{1},\cdots,i_{n}$, $j_{1},\cdots,j_{n} \in \left\{0, 1\right\}$, let
\begin{align*}
f_{\beta}(z)=f^{(i_{1})}_{\beta_{1}}(z_{1})\cdots f^{(i_{n})}_{\beta_{n}}(z_{n}),\quad
g_{\gamma}(z)=g^{(j_{1})}_{\gamma_{1}}(z_{1})\cdots g^{(j_{n})}_{\gamma_{n}}(z_{n}).
\end{align*}
Then the Hankel product $H^{*}_{f_{\beta}}H_{g_{\gamma}}$
is bounded on $F^{2,m}$  if and only if  at least one of the following conditions holds:

$(1)$ $k=(0,\cdots,0)$ and $\beta_{s}=0$ for any $1\leq s\leq n$  such that $i_{s}=1$.

$(2)$ $l=(0,\cdots,0)$ and $\gamma_{t}=0$ for any $1\leq t\leq n$  such that $j_{t}=1$.

$(3)$ $n=\beta_{1}=\gamma_{1}=i_{1}=j_{1}=1$ and $k_{1}=l_{1}=0$.
\end{lemma}
\begin{proof}
To begin with, we use the same notations $\theta$, $\vartheta$, $\varphi$ and $\psi$ as in Lemma \ref{ToeL2}. Then
by Lemma \ref{HanL1}, for any $\alpha\in\mathbb{N}^{n}$ satisfying $\alpha \geq\beta+\gamma$,
\begin{align*}
&\hs\hs H^{*}_{f_{\beta}}H_{g_{\gamma}}e_{\alpha}\\
&=\sum_{\mu_{1}\leq k_{1}}\cdots\sum_{\mu_{n}\leq k_{n}}
\sum_{\nu_{1}\leq l_{1}}\cdots\sum_{\nu_{n}\leq l_{n}}
a^{(i_{1})}_{\mu_{1}}\cdots a^{(i_{n})}_{\mu_{n}}
b^{(j_{1})}_{\nu_{1}}\cdots b^{(j_{n})}_{\nu_{n}}
H^{*}_{z^{\theta}\bar{z}^{\vartheta}}
H_{z^{\varphi}\bar{z}^{\psi}}e_{\alpha}\\
&=\sum_{\mu_{1}\leq k_{1}}\cdots\sum_{\mu_{n}\leq k_{n}}
\sum_{\nu_{1}\leq l_{1}}\cdots\sum_{\nu_{n}\leq l_{n}}
a^{(i_{1})}_{\mu_{1}}\cdots a^{(i_{n})}_{\mu_{n}}b^{(j_{1})}_{\nu_{1}}\cdots b^{(j_{n})}_{\nu_{n}}
B^{\theta\vartheta\varphi\psi}_{\alpha}
e_{\alpha+\vartheta+\varphi-\theta-\psi},
\end{align*}
where
\begin{align*}
\nonumber B^{\theta\vartheta\varphi\psi}_{\alpha}
&:=\bigg(\frac{(\alpha+\vartheta+\varphi)!(m+n-1+|\alpha+\vartheta+\varphi|)!}
{\alpha !(n-1+|\alpha+\vartheta+\varphi|)! }-\frac{(\alpha+\varphi)!(\alpha+\vartheta+\varphi-\psi)!}
{\alpha !(\alpha+\varphi- \psi) ! }\\
\nonumber&\hs\hs\times\frac{(m+n-1+|\alpha+\varphi|)!(n-1+|\alpha +\varphi- \psi|) !(m+n-1+|\alpha+\vartheta+\varphi-\psi|)!}
{(n-1+|\alpha+\varphi|)!(m+n-1+|\alpha +\varphi- \psi|) !(n-1+|\alpha+\vartheta+\varphi-\psi|)! }\bigg)\\
&\hs\hs\times\sqrt{\frac{\alpha !(n-1+|\alpha|) !(n-1+|\alpha+\vartheta+\varphi- \theta - \psi|) !}{(\alpha+\vartheta+\varphi- \theta - \psi) !(m+n-1+|\alpha|) !(m+n-1+|\alpha+\vartheta+\varphi- \theta - \psi|) !}}.
\end{align*}
If $\vartheta\neq0$ and $\psi\neq0$, then by Lemma \ref{HanL1} again, we have
$B^{\theta\vartheta\varphi\psi}_{\alpha}\neq0$ and
\begin{align*}
B^{\theta\vartheta\varphi\psi}_{\alpha}
&\sim\sum_{s=1}^{n}\vartheta_{s}\psi_{s}\alpha_{s}^{-1}
\alpha^{\frac{\theta+\vartheta+\varphi+\psi}{2}}\\
&=\sum_{s=1}^{n}\left(\mu_{s}+\chi_{\{1\}}(i_{s})\beta_{s}\right)
\left(\nu_{s}+\chi_{\{1\}}(j_{s})\gamma_{s}
\right)\alpha_{s}^{-1}
\alpha^{\frac{\beta+\gamma}{2}+\mu+\nu}.
\end{align*}
Since $a_{k_{i}}$, $b_{l_{i}}$ are nonzero constants for each $i=1,\cdots,n$, we have
\begin{align*}
&\left\|H^{*}_{f_{\beta}}H_{g_{\gamma}}e_{\alpha}\right\|_{2, m}\\
=&\left|\sum_{\mu_{1}\leq k_{1}}\cdots\sum_{\mu_{n}\leq k_{n}}
\sum_{\nu_{1}\leq l_{1}}\cdots\sum_{\nu_{n}\leq l_{n}}
a^{(i_{1})}_{\mu_{1}}\cdots a^{(i_{n})}_{\mu_{n}}b^{(j_{1})}_{\nu_{1}}\cdots b^{(j_{n})}_{\nu_{n}}
B^{\theta\vartheta\varphi\psi}_{\alpha}\right|\\
\sim&\left|a^{(i_{1})}_{k_{1}}\cdots a^{(i_{n})}_{k_{n}}b^{(j_{1})}_{l_{1}}\cdots b^{(j_{n})}_{l_{n}}\right|
\sum_{s=1}^{n}\left(k_{s}+\chi_{\{1\}}(i_{s})\beta_{s}\right)
\left(l_{s}+\chi_{\{1\}}(j_{s})\gamma_{s}\right)
\alpha_{s}^{-1}\alpha^{\frac{\beta+\gamma}{2}+k+l}
\end{align*}
for $\alpha\geq\beta+\gamma$. Since the Hankel product $H^{*}_{f_{\beta}}H_{g_{\gamma}}$ is bounded on $F^{2,m}$  if and only if the sequence
$$\left\{\left\|H^{*}_{f_{\beta}}H_{g_{\gamma}}e_{\alpha}\right\|_{2, m}\right\}_{\alpha\geq\beta+\gamma}$$
is bounded.
Therefore, $H^{*}_{f_{\beta}}H_{g_{\gamma}}$ is bounded on $F^{2,m}$ if and only if the following expression
\begin{align*}
\left(k_{s}+\chi_{\{1\}}(i_{s})\beta_{s}\right)
\left(l_{s}+\chi_{\{1\}}(j_{s})\gamma_{s}\right)
\alpha_{s}^{-1}\alpha^{\frac{\beta+\gamma}{2}+k+l}
\end{align*}
is independent of $\alpha$ for each $s=1,\cdots,n$, which is equivalent to that at least one of the following statements holds:

$(a)$ $\left(k_{s}+\chi_{\{1\}}(i_{s})\beta_{s}\right)
\left(l_{s}+\chi_{\{1\}}(j_{s})\gamma_{s}\right)=0$ for each $s=1,\cdots,n$.

$(b)$ $n=\beta_{1}=\gamma_{1}=i_{1}=j_{1}=1$ and $k_{1}=l_{1}=0$.\\
Since  $(a)$ is equivalent to condition $(1)$ or $(2)$, the desired result is then obtained.
\end{proof}

We proceed to prove the main theorem in this section.

\begin{proof}[\pmb{Proof of Theorem \ref{HP}}]
If the statement $(1)$ or $(2)$ is true, then $H^{*}_{f}=0$ or $H_{g}=0$, it follows that   $H^{*}_{f}H_{g}$ is bounded on $F^{2,m}$. If the statement $(3)$ is true, then we have
\begin{align*}
H^{*}_{f}H_{g}e_{\alpha}
&=\overline{a}bH^{*}_{\overline{z}}H_{\overline{z}}e_{\alpha}\\
&=\overline{a}be_{\alpha}
\end{align*}
by Lemma \ref{HanL2}, which implies that the Hankel product $H^{*}_{f}H_{g}$ is bounded on $F^{2,m}$.

Conversely, assume the Hankel product $H^{*}_{f}H_{g}$ is bounded on $F^{2,m}$. If neither $f$ nor $g$ is holomorphic, we are to show that the statement $(3)$ must be true. Since $f$ is a polynomial in $z,\overline{z}\in\mathbb{C}^{n}$, there exist $k=(k_{1},\cdots,k_{n})$ and
$l=(l_{1},\cdots,l_{n})\in\mathbb{N}^{n}$ such that
\begin{align}\label{fzz1}
f(z,\overline{z})=\prod_{s=1}^{n}\left(\sum_{\beta_{s}\leq k_{s}}\sum_{\gamma_{s}\leq l_{s}}a_{\beta_{s}\gamma_{s},s}z_{s}^{\beta_{s}}\overline{z}_{s}^{\gamma_{s}}\right).
\end{align}
Let
\begin{align*}
f_{1}(z,\overline{z})=\prod_{s=1}^{n}\left(\sum_{\beta_{s}\leq k_{s}}a_{\beta_{s}0,s}z_{s}^{\beta_{s}}\right).
\end{align*}
Then $f_{1}$ is said to be the pure holomorphic part of $f$. Similarly, denote $g_{1}$ by the pure holomorphic part of $g$. Let $f_{2} = f -f_{1}$ and $g_{2} = g -g_{1}$. Then by our assumption, we see that neither $f_{2}$ nor $g_{2}$ is 0. Moreover, from the discussion before Theorem \ref{TP},  $f_{2}$ and $g_{2}$ admit  expansions
\begin{align*}
f_{2}=\prod_{s=1}^{n}\left(\sum_{\theta_{s}=i_{0,s}}^{i_{1,s}}
F_{\theta_{s}}\right),\quad
g_{2}=\prod_{t=1}^{n}\left(\sum_{\tau_{t}=j_{0,t}}^{j_{1,t}}
G_{\tau_{t}}\right),
\end{align*}
where $F_{i_{0,s}}$, $F_{i_{1,s}}$, $G_{j_{0,s}}$ and $G_{j_{1,s}}$ are nonzero. Therefore,
\begin{align}
\nonumber H^{*}_{f}H_{g}e_{\alpha}&=H^{*}_{f_{2}}H_{g_{2}}e_{\alpha}\\
\nonumber&=\sum_{\theta_{1}=i_{0,1}}^{i_{1,1}}\cdots
\sum_{\theta_{n}=i_{0,n}}^{i_{1,n}}
\sum_{\tau_{1}=j_{0,1}}^{j_{1,1}}\cdots\sum_{\tau_{n}=j_{0,n}}^{j_{1,n}}
H^{*}_{F_{\theta_{1}}\cdots F_{\theta_{n}}}
H_{G_{\tau_{1}}\cdots G_{\tau_{n}}}e_{\alpha}\\
\label{(HP1)}&=H^{*}_{F_{i_{0,1}}\cdots F_{i_{0,n}}}
H_{G_{j_{1,1}}\cdots G_{j_{1,n}}}e_{\alpha}
+\sum_{\substack{(\theta_{1},\cdots,\theta_{n},\tau_{1},\cdots,\tau_{n})
\neq\\(i_{0,1},\cdots,i_{0,n},j_{1,1},\cdots,j_{1,n})}}
H^{*}_{F_{\theta_{1}}\cdots F_{\theta_{n}}}
H_{G_{\tau_{1}}\cdots G_{\tau_{n}}}e_{\alpha}
\end{align}
 for any $\alpha\in \mathbb{N}^{n}$.
Set multi-index
$$
\kappa=\left(\max \left\{\left|i_{0,1}\right|,\left|i_{1,1}\right|\right\}+\max \left\{\left|j_{0,1}\right|,\left|j_{1,1}\right|\right\}
,\cdots,\max \left\{\left|i_{0,n}\right|,\left|i_{1,n}\right|\right\}+\max \left\{\left|j_{0,n}\right|,\left|j_{1,n}\right|\right\}\right).
$$
It follows from the definitions of $F_{\theta_{s}}$, $G_{\tau_{t}}$ and the proof of Lemma \ref{HanL2} that for any $\alpha\geq\kappa$, $\beta=\left(\theta_{1},\cdots,\theta_{n}\right),\
\gamma=\left(\tau_{1},\cdots,\tau_{n}\right)$ with $i_{0,s} \leq \theta_{s} \leq i_{1,s}$ and $j_{0,t} \leq \tau_{t}\leq j_{1,t}$ ($s,t=1,\cdots,n$) satisfying $(\theta_{1},\cdots,\theta_{n},\tau_{1},\cdots,\tau_{n})
\neq(i_{1,1},\cdots,i_{1,n},j_{1,1},\cdots,j_{1,n})$, we have
$$
H^{*}_{F_{\theta_{1}}\cdots F_{\theta_{n}}}
H_{G_{\tau_{1}}\cdots G_{\tau_{n}}}e_{\alpha}\in \text{Span}\{e_{\alpha+\gamma-\beta}\}.
$$
But the first term of (\ref{(HP1)})
$$H^{*}_{F_{i_{0,1}}\cdots F_{i_{0,n}}}
H_{G_{j_{1,1}}\cdots G_{j_{1,n}}}e_{\alpha}\in \text{Span}\{e_{\alpha+\gamma^{\prime}-\beta^{\prime}}\},$$
where $\gamma^{\prime}=(i_{0,1},\cdots,i_{0,n})$ and $\beta^{\prime}=(j_{1,1},\cdots,j_{1,n})$.
Therefore,  we conclude that $H^{*}_{F_{i_{0,1}}\cdots F_{i_{0,n}}}
H_{G_{j_{1,1}}\cdots G_{j_{1,n}}}e_{\alpha}$ is orthogonal to the second term of (\ref{(HP1)}) for $\alpha \geq\kappa$.
This makes
\begin{align*}
\left\|H^{*}_{f}H_{g}e_{\alpha}\right\|_{2,m}
\geq\left\|H^{*}_{F_{i_{0,1}}\cdots F_{i_{0,n}}}
H_{G_{j_{1,1}}\cdots G_{j_{1,n}}}e_{\alpha}\right\|_{2,m}
\end{align*}
for $\alpha \geq\kappa$.
Carefully examining the proof of Lemma \ref{HanL2}, we see that  $H^{*}_{F_{i_{1,1}}\cdots F_{i_{1,n}}}
H_{G_{j_{1,1}}\cdots G_{j_{1,n}}}e_{\alpha}$ is bounded on $F^{2,m}$
if and only if the sequence
$$\left\{\left\|H^{*}_{F_{i_{0,1}}\cdots F_{i_{0,n}}}
H_{G_{j_{1,1}}\cdots G_{j_{1,n}}}e_{\alpha}\right\|_{2,m}\right\}_{\alpha\geq\kappa}$$
is bounded on $F^{2,m}$.

Notice from the definitions of $f_{2}$ and $g_{2}$ that, for $\theta_{s} \geq 0$ (resp. $\tau_{t} \geq 0$), $F_{\theta_{s}}$ (resp. $G_{\tau_{t}}$) doesn't contain any term as $a_{\theta_{s}}z_{s}^{\theta_{s}}$ (resp. $b_{\tau_{t}}z_{t}^{\tau_{t}}$),
where $a_{\theta_{s}}$ (resp.$b_{\tau_{t}}$) denotes the coefficient. In other words, for $\theta_{s} \geq 0$ (resp.$\tau_{t} \geq 0$), the term $F_{\theta_{s}}$  (resp. $G_{\tau_{t}}$) is of the following form:
\begin{align}\label{(HP3)}
\sum_{1\leq\mu_{s}\leq k_{s}}
a_{\mu_{s},s}z_{s}^{\mu_{s}+\theta_{s}}\overline{z}_{s}^{\mu_{s}},\quad
\left(\text{resp}.\sum_{1\leq\nu_{t}\leq l_{t}}
b_{\nu_{t},t}z_{t}^{\nu_{t}+\tau_{t}}\overline{z}_{t}^{\nu_{t}},\right)
\end{align}
where $k_{s}$ and $l_{t}$ are positive integers greater than or equal to 1.

For $\theta_{s} < 0$ (resp.$\tau_{t}< 0$), the term $F_{\theta_{s}}$  (resp. $G_{\tau_{t}}$)  is of the following form:
\begin{align}\label{(HP4)}
\sum_{\mu_{s}\leq k_{s}}
a_{\mu_{s},s}z_{s}^{\mu_{s}}\overline{z}_{s}^{\mu_{s}+|\theta_{s}|},\quad
\left(\text{resp}.\sum_{\nu_{t}\leq l_{t}}
b_{\nu_{t},t}z_{t}^{\nu_{t}}\overline{z}_{t}^{\nu_{t}+|\gamma_{t}|},\right)
\end{align}

If $i_{0,s}\geq0$ or $j_{1,t}\geq0$ for all $s,t=1,\cdots,n$, then it follows   from \eqref{(HP3)} and Lemma \ref{HanL2} that $H^{*}_{F_{i_{0,1}}\cdots F_{i_{0,n}}}H_{G_{j_{1,1}}\cdots G_{j_{1,n}}}$ is unbounded. Thus, the boundedness of $H^{*}_{F_{i_{0,1}}\cdots F_{i_{0,n}}}
H_{G_{j_{1,1}}\cdots G_{j_{1,n}}}$  implies that $i_{0,s}<0$ and $j_{1,t}<0$ for all $s,t=1,\cdots,n$. Then $F_{\theta_{s}}$ is the form of \eqref{(HP4)}. It follows from $(3)$ of Lemma \ref{HanL2}, we have $n=1$ and $F_{i_{0,1}} = a_{0}\overline{z}$, $G_{j_{1,1}}=b_{0}\overline{z}$, where are $a_{0}$, $b_{0}$ nonzero constants and $\overline{z}\in\mathbb{C}$.

As discussed above, we can also conclude that the Hankel product is  bounded if (\ref{(HP1)}) is replaced  by
\begin{align*}
H^{*}_{f}H_{g}e_{\alpha}
=H^{*}_{F_{i_{1,1}}\cdots F_{i_{1,n}}}
H_{G_{j_{0,1}}\cdots G_{j_{0,n}}}e_{\alpha}
+\sum_{\substack{(\theta_{1},\cdots,\theta_{n},\tau_{1},\cdots,\tau_{n})
\neq\\(i_{1,1},\cdots,i_{1,n},j_{0,1},\cdots,j_{0,n})}}
H^{*}_{F_{\theta_{1}}\cdots F_{\theta_{n}}}
H_{G_{\tau_{1}}\cdots G_{\tau_{n}}}e_{\alpha}.
\end{align*}
Similar to the discussion of $H^{*}_{F_{i_{0,1}}\cdots F_{i_{0,n}}}
H_{G_{j_{1,1}}\cdots G_{j_{1,n}}}e_{\alpha}$, we can also conclude that  $n=1$ and $F_{i_{1,1}} = a^{\prime}_{0}\overline{z}$, $G_{j_{0,1}}=b^{\prime}_{0}\overline{z}$, where $a^{\prime}_{0}$, $b^{\prime}_{0}$ are nonzero constants and $\overline{z}\in\mathbb{C}$.
Therefore, $f_{2}(z)= a\overline{z}$ and $g_{2}(z)=b\overline{z}$, where  $a$ and $b$ are nonzero constants  and $\overline{z}\in\mathbb{C}$, hence the statement $(3)$ is true. This completes the proof.
\end{proof}

\begin{corollary}
Assume that $f$ is a polynomial in $z,\overline{z}\in\mathbb{C}^{n}$. Then the Hankel operator
$H_{f}$ is bounded on $F^{2,m}$ if and only if one of the following statements is true

$(1)$  $f$ is holomorphic.

$(2)$ $n=1$ and there exists a holomorphic polynomial $f_{1}$ such that
\begin{align*}
f=f_{1}+a\overline{z},
\end{align*}
where $a$ is a constant and $z\in\mathbb{C}$.
\end{corollary}
\begin{proof}
It is a direct consequence of Theorem \ref{HP} by setting $g = f$.
\end{proof}
\begin{corollary}
Assume that $f$ is a polynomial in $z,\overline{z}\in\mathbb{C}^{n}$.  Then the Hankel operator $H_{f}$ is compact on $F^{2,m}$ if and only if $f$ is holomorphic.
\end{corollary}

\end{document}